\documentclass[onecolumn,draft]{IEEEtran}

\usepackage{graphics}
\usepackage{amsmath, amsthm, amssymb}

\newtheorem{theo}{Theorem}[section]
\newtheorem{lemm}[theo]{Lemma}
\newtheorem{coro}[theo]{Corollary}
\newtheorem{prop}[theo]{Proposition}
\newtheorem{defi}[theo]{Definition}

\newtheorem{example}[theo]{Example}

\newcommand{\lland}{\mbox{ and }}

\newcommand{\w}{\mbox{\rm wt}}

\newcommand{\ba}{{\mathbf a}}
\newcommand{\bb}{{\mathbf b}}
\newcommand{\bu}{\mathbf{u}}

\newcommand{\be}{\mathbf{e}}

\newcommand{\one}{{\mathbf 1}}
\newcommand{\Z}{\mathbb{Z}}
\newcommand{\cA}{{\mathcal{A}}}
\newcommand{\cC}{{\mathcal{C}}}
\newcommand{\cD}{{\mathcal{D}}}
\newcommand{\cQ}{{Q}}

\newcommand{\cG}{{\mathcal{G}}}

\newcommand{\cR}{{\mathcal{R}}}

\newcommand{\rB}{{\overline{r}}}

\newcommand{\tauB}{{\overline{\tau}}}
\newcommand{\Supp}{\operatorname{Supp}}
\newcommand{\ccolon}{\colon\!}

\newcommand{\gen}[1]{\langle #1 \rangle}
\newcommand{\genset}{$x_1,\dots,x_{\sigma}$; $r_1,\dots,r_{\tau}$; $s_1,s_{\upsilon}$~}

\newcommand{\zz}{\Z_2\Z_4}
\newcommand{\qq}{\Z_2\Z_4\cQ_8}

\title{Hadamard $\qq$-codes. Constructions based on the rank and dimension of the kernel.\thanks{This work has been
partially supported by the Spanish MICINN grant TIN2013-40524-P and the Catalan grant 2014SGR-691.}}

\author{P. Montolio and J. Rif\`{a} %
\thanks{P. Montolio is with the Computing, Multimedia and Telecommunication Studies, Universitat Oberta de Catalunya.} %
\thanks{J.~Rif\`{a} is with the Department of Information and Communications Engineering, Universitat Aut\`{o}noma de Barcelona.}
}

\begin{document}
\maketitle

\begin{abstract}
This work deals with Hadamard $\Z_2\Z_4\cQ_8$-codes, which are binary codes after a Gray map from a subgroup of the direct product of $\Z_2$, $\Z_4$ and $\cQ_8$ groups, where $\cQ_8$ is the quaternionic group. In a previous work, these kind of codes were classified in five shapes. In this paper we analyze the allowable range of values for the rank and dimension of the kernel, which depends on the particular shape of the code. We show that all these codes can be represented in a standard form, from a set of generators, which help to a well understanding of the characteristics of each shape. The main results are the characterization of Hadamard $\qq$-codes as a quotient of a semidirect product of $\Z2\Z4$-linear codes and, on the other hand, the construction of  Hadamard $\qq$-codes code with any given pair of allowable parameters for the rank and dimension of the kernel.
\end{abstract}

\section{Introduction}

Non-linear codes (like $\zz$-linear codes and $\qq$-codes) have received a great deal of attention since \cite{z4}. The codes this paper deals with can be characterized as the image of a subgroup, by a suitable Gray map, of an algebraic group like the direct product of $\Z_2$, $\Z_4$ and $\cQ_8$, the quaternionic group of order 8 \cite{rp}.

Hadamard matrices with a subjacent algebraic structure have been deeply studied, as well as the links with other topics in algebraic combinatorics or applications \cite{kar}. We quote just a few papers about this subject \cite{ito,flan,lfh}, where we can find beautiful equivalences between Hadamard groups, 2-cocyclic matrices and relative difference sets. On the other hand, from the side of coding theory, it is desirable that the algebraic structures we are dealing with preserves the Hamming distance. This is the case, for example, of the $\zz$-linear codes which has been intensively studied during the last years~\cite{z4,fv}. More generally, the propelinear codes and, specially those which are translation invariant, are particularly interesting because the subjacent group structure has the property that both, left and right product, preserve the Hamming distance. Translation invariant propelinear codes has been characterized as the image of a subgroup by a suitable Gray map of a direct product of $\Z_2$, $\Z_4$ and $\cQ_8$~\cite{rp}.

In this paper we analyze codes that have both properties, being Hadamard and $\Z_2\Z_4\cQ_8$-codes. These codes were previously studied and classified~\cite{rio-2013} in five shapes. The aim of this paper is to go further. First of all by giving an standard form for a set of generators of the code, depending on the parameters, which helps to understand of the characteristics of each shape and then by putting the focus in an exact computation of the values of the rank and dimension of the kernel.
One of the main results is to characterize the $\qq$-codes as a quotient of a semidirect product of Hadamard $\Z2\Z_4$-linear codes. The second main result is to construct, using the above characterization, Hadamard $\qq$-codes whose values for the rank and dimension of the kernel are any allowable pair previously chosen.

The structure of the paper is as follows. Section~\ref{sec:preliminaries} introduces the notation, the basic classification from~\cite{rio-2013} and preliminary concepts; Section \ref{sec:standard} shows the standard form of generators that allows to represent any Hadamard $\Z_2\Z_4\cQ_8$-code in a unique way, this section finishes with two important theorems which characterizes a Hadamard $\qq$-code as a quotient of a semidirect product of $\Z_2\Z_4$-linear codes (Theorems~\ref{main} and \ref{prop:z2z4}); Section \ref{sec:kernelrank} studies the values of the rank and dimension of the kernel, depending on the shape and parameters of the $\qq$-code and in Section~\ref{sec:construction} we give the constructions of $\qq$-codes fulfilling the requirements for the prefixed values of  the dimension of the kernel and rank. We finish this last section with Theorem~\ref{fites}, a compendium of the results reached in this section, and a couple of examples about the constructions and achievement of codes with all allowable pair for the values of rank and dimension of the kernel.

\section{Preliminaries}\label{sec:preliminaries}

Almost all the definitions and concepts below, in these preliminaries, can be found in \cite{rio-2013}.

Let $\Z_2$ and $\Z_4$ denote the binary field and the ring of integers modulo 4, respectively. Any non-empty subset of $\Z_2^n$ is called a binary code and a linear subspace of $\Z_2^n$ is called a \emph{binary linear code} or a \emph{$\Z_2$-linear code}. Let $\w(v)$ denote the {\em Hamming weight} of a
vector $v \in \Z_2^n$ (i.e., the number of its nonzero components),
and let $d(v, u)=\w(v+u)$, the {\em Hamming distance} between two
vectors $v,u\in\Z_2^n$.

Let $\cQ_8$ be the {\em quaternionic group} on eight elements. The following equalities provides a presentation and the list of elements of $\cQ_8$:
\begin{equation*}
\begin{split}
\cQ_8=&\gen{\ba,\bb\, :\,\ba^4=\ba^2\bb^2=\one,\bb\ba\bb^{-1}=\ba^{-1}} =
\{\one,\ba,\ba^2,\ba^3,\bb,\ba\bb,\ba^2\bb,\ba^3\bb\}.
\end{split}
\end{equation*}

Given three non-negative integers $k_1$, $k_2$ and $k_3$, denote as $\cG$ the group $\Z_2^{k_1}\times\Z_4^{k_2}\times\cQ_8^{k_3}$. Any element of $\cG$ can be represented as a vector where the first $k_1$ components belong to $\Z_2$, the next $k_2$ components belong to $\Z_4$ and the last $k_3$ components belong to $\cQ_8$.

We use the multiplicative notation for $\cG$ and we denote by $\be$ the identity element of the group and by $\bu$ the element of order two: $\be=(0,\stackrel{k_1+k_2}{\dots},0,\one,\stackrel{k_3}{\dots},\one)$ and $\bu=(1,\stackrel{k_1}{\dots},1,2,\stackrel{k_2}{\dots},2,\ba^2,\stackrel{k_3}{\dots},\ba^2)$.

We call {\em Gray map} the function $\Phi$:
$$
\Phi:\;\Z_2^{k_1}\times\Z_4^{k_2}\times\cQ_8^{k_3}\;\longrightarrow\;\Z_2^{k_1+2k_2+4k_3},
$$
acting componentwise in such a way that over the binary part is the identity, over the quaternary part acts as the usual Gray map, so $0 \rightarrow (00)$, $1 \rightarrow (01)$, $2\rightarrow(11)$, $3\rightarrow(10)$ and over the quaternionic part acts in the following way~\cite{rio-2013}:
$\one\rightarrow(0,0,0,0)$,    $\bb\rightarrow(0,1,1,0)$, $\ba\rightarrow(0,1,0,1)$,   $\ba\bb\rightarrow(1,1,0,0)$,
$\ba^2\rightarrow(1,1,1,1)$, $\ba^2\bb\rightarrow(1,0,0,1)$,  $\ba^3\rightarrow(1,0,1,0)$,  $\ba^3\bb\rightarrow(0,0,1,1)$.

Note that $\Phi(\be)$ is the all-zeros vector and $\Phi(\bu)$ is the all-ones vector.

Binary codes $C=\Phi(\cC)$ are called $\qq$-codes. In the specific case $k_3=0$, $C$ is called $\Z_2Z_4$-linear code. In this last case, note that $\cC =\Z_2^\gamma\times \Z_4^\delta \subset \Z_2^{k_1}\times \Z_4^{k_2}$. We will say that $\cC$ is of type $2^{\gamma} 4^\delta$ \cite{z4}.

We are interested in Hadamard binary codes $C=\Phi(\cC)$ where $\cC$ is a subgroup of $\cG=\Z_2^{k_1}\times\Z_4^{k_2}\times\cQ_8^{k_3}$ of length $n=2^m$. All through the paper we are assuming it.

The {\em kernel} of a binary code $C$ of length $n$ is $K(C)=\{z \in \Z_2^n : C + z = C \}$.
The dimension of $K(C)$ is denoted by $k(C)$ or simply $k$.
The {\em rank} of a binary code $C$ is the dimension of the linear span of $C$. It is denoted by $r(C)$ or simply $r$.

A \emph{Hadamard matrix} of order $n$ is a matrix of size $n \times n$ with entries $\pm 1$, such that $HH^T=n I$. Any two rows (columns) of a Hadamard matrix agree in precisely $n/2$ components. If $n>2$ then any three rows (columns) agree in precisely $n/4$ components. Thus, if $n>2$ and there is a Hadamard matrix of order $n$ then $n$ is multiple of 4.

Two \emph{Hadamard matrices} are \emph{equivalent} if one can be obtained from the other by permuting rows and/or columns and multiplying rows and/or columns by $-1$. With the last operations we can change the first
row and column of $H$ into $+1$'s and we obtain an equivalent Hadamard matrix which is called normalized. If $+1$'s are replaced by $0$'s and $-1$'s by $1$'s, the initial Hadamard matrix is changed into a (binary) Hadamard matrix and, from now on, we will refer to it when we deal with Hadamard matrices. The binary code consisting of the rows of a (binary) Hadamard matrix and their complements is called a (binary) \emph{Hadamard code}, which is of length $n$, with $2n$ codewords, and minimum distance $n/2$.

Hadamard $\qq$-codes were studied in~\cite{rio-2013} and a classification in five shapes was given.  Set $|T(\cC)|=2^{\sigma}$, $|Z(\cC)/T(\cC)|=2^{\delta}$ and $|\cC/ Z(\cC)|=2^{\rho}$, where $T(\cC)$ is the subgroup of elements of order two, $Z(\cC)$ is the center of $\cC$ and $m = \sigma+\delta+\rho-1$. A normalized generator set in \cite{rio-2013} has the form $\cC=\gen{x_1,\dots,x_{\sigma};$ $ y_1, \dots,y_{\delta};$ $z_1,$ $\dots,z_{\rho}}$, where $x_i$ are elements of order two generating $T(\cC)=\gen{x_1...x_\sigma}$ and $Z(\cC)=\gen{x_1,\dots,x_{\sigma}$; $y_1,\dots,y_{\delta}}$ is the center of $\cC$. In summary, the five shapes found in~\cite{rio-2013} are:
\begin{itemize}\label{shapes}
\item{Shape 1:} $\rho=0$. However, in the current paper we distinguish the case when $\bu$ is not the square of some element of order four and we call these codes of shape $1$, and the case when $\bu$ is the square of some element of order four and we call these codes of shape $1^*$.
\item{Shape 2:} $\delta=0$, $z_1^2=z_2^2=[z_1,z_2]=\bu$, $[z_i,z_j]=z_j^2$ and $[z_j,z_k]=\be$ for every $i\in\{1,2\}$ and $3\le j,k \le \rho$.
\item{Shape 3:} $\delta=0$, $z_1^2=\bu\not\in \gen{z_2^2,\dots,z_{\rho}^2}\cong \Z_2^{\rho-1}$, $[z_1,z_i]=z_i^2$ and $[z_i,z_j]=\be$, for every $i\not=j$ in $\{2,\dots,\rho\}$.
\item{Shape 4:} $\delta \leq 1$ and $z_1^2=z_2^2=[z_1,z_2] \not=\bu$. However, as for shape $1$, in the current paper we distinguish the case when $\bu$ is not the square of some element of order four (this is equivalent to $\delta=0$) and we call these codes of shape $4$, and the case when $\bu$ is the square of some element of order four (this is equivalent to $\delta=1$) and we call these codes of shape $4^*$.
\item{Shape 5:} $\delta=0$, $\rho=4$, $z_1^2=z_2^2=[z_1,z_2]=\bu\ne z_3^2=z_4^2=[z_3,z_4]$ and $[z_i,z_j] \in \langle z_j^2\rangle$ for every $i\in \{1,2\}$ and $j\in\{3,4\}$.
\end{itemize}

Two elements $a$ and $b$ of $\cC$ commutes if and only if $ab=ba$. As an extension of this concept, the {\em commutator} of $a$ and $b$ is defined as the element $[a,b]$ such that $ab=[a,b]ba$.  Note that all commutators belong to $T(\cC)$ and any element of $T(\cC)$ commutes with all elements of $\cC$.

We say that two elements $a$ and $b$ of $\cC$ {\em swap} if and only if $\Phi(ab)=\Phi(a)+\Phi(b)$. As an extension of this concept, define the {\em swapper} of $a$ and $b$ as the element $(a\ccolon b)$ such that $\Phi((a\ccolon b)ab)=\Phi(a)+\Phi(b)$. Note that all swappers belong to $T(\cG)$ but they can be out of $\cC$.

Both, commutators and swappers can be obtained as a component-wise expression, if $a=(a_1,\dots,a_l)$ and $b=(b_1,\dots,b_l)$ then $(a\ccolon b) = ((a_1\ccolon b_1),\dots,(a_l\ccolon b_l))$ and $[a,b] = ([a_1,b_1],\dots,[a_l,b_l])$. Table~\ref{table:1} and Table~\ref{table:2} describes the values of all swappers and commutators, respectively, in $\Z_4$ and $\cQ_8$ (the value in $\Z_2$ is always 0).

\begin{table}[ht]
\begin{center}
\begin{tabular}{c|cc}
& 0,2 & 1,3 \\\hline
0,2 & 0 & 0 \\
1,3 & 0 & 2
\end{tabular}
\begin{tabular}{c|cccc}
& $\one$,$\ba^2$ & $\ba$,$\ba^3$ & $\bb$,$\ba^2\bb$ & $\ba\bb$,$\ba^3\bb$ \\\hline
$\one$,$\ba^2$ & $\one$ & $\one$ & $\one$ & $\one$\\
$\ba$,$\ba^3$  & $\one$ & $\ba^2$ & $\ba^2$ & $\one$\\
$\bb$,$\ba^2\bb$ & $\one$ & $\one$ & $\ba^2$ & $\ba^2$\\
$\ba\bb$,$\ba^3\bb$ & $\one$ & $\ba^2$ & $\one$ & $\ba^2$
\end{tabular}
\end{center}
\caption{Swappers in $\Z_4$ and $\cQ_8$}\label{table:1}
\end{table}

\begin{table}[ht]
\begin{center}
\begin{tabular}{c|cc}
& 0,2 & 1,3 \\\hline
0,2 & 0 & 0 \\
1,3 & 0 & 0
\end{tabular}
\begin{tabular}{c|cccc}
& $\one$,$\ba^2$ & $\ba$,$\ba^3$ & $\bb$,$\ba^2\bb$ & $\ba\bb$,$\ba^3\bb$ \\\hline
$\one$,$\ba^2$ & $\one$ & $\one$ & $\one$ & $\one$\\
$\ba$,$\ba^3$  & $\one$ & $\one$ & $\ba^2$ & $\ba^2$\\
$\bb$,$\ba^2\bb$ & $\one$ & $\ba^2$ & $\one$ & $\ba^2$\\
$\ba\bb$,$\ba^3\bb$ & $\one$ & $\ba^2$ & $\ba^2$ & $\one$
\end{tabular}
\end{center}
\caption{Commutators in $\Z_4$ and $\cQ_8$}\label{table:2}
\end{table}

It is known~\cite{rio-2013} the following relationship between swappers and the kernel. For any element $a$ of $\cC$ we have $\Phi(a) \in K(C)$ if and only if all swappers $(a\ccolon b) \in \cC$, for every $b \in \cC$. Moreover, the linear span of $C$ can be seen as $\Phi(\langle \cC \cup S(\cC)\rangle)$, where $\langle \cC \cup S(\cC)\rangle$ is the group generated by $\cC$ and $S(\cC)$, the set of swappers of the elements in $\cC$.

Using Table-\ref{table:1} and Table-\ref{table:2} it can be easily verified the next lemma.
\begin{lemm}\label{previs}
For any $a,b,c \in \cG$:
\begin{enumerate}
\item $[a,b]=[b,a]$. Note it is not always true that $(a\ccolon b)=(b\ccolon a)$.
\item $(ab\ccolon c)=(a\ccolon c)(b\ccolon c)$ and $(c\ccolon ab)=(c\ccolon a)(c\ccolon b)$
\item $[ab,c]=[a,c][b,c]$.
\item $(a\ccolon b)(b\ccolon a)=[a,b]$
\item $(a\ccolon a)=a^2$
\item if $a^2=\be$ then $[a,b]=(a\ccolon b)=(b\ccolon a)=\be$.
\item if $a^2=\bu \lland [a,b]=\be$ then $(a\ccolon b)=(b\ccolon a)=b^2$.
\end{enumerate}
\end{lemm}



\begin{defi}
For any $x \in T(\cG)$ define $M(x)$ as the set of components where $x$ has, as entry, an element of order two, $\varnothing \subseteq M(x) \subseteq M(\bu)$.
\end{defi}
Example: let $x=(1,a^2,a^2,1,1,a^2)$ then $M(x)=\{1,2,5\}$, where we enumerate the $0$th component as the first one.

The next results in this section are technical lemmas, which prove to be useful later.


\begin{lemm}\label{lem:2}~
Let $x,y \in \cG$, then
\begin{enumerate}
\item $M(\,(x\ccolon y)\,) \subseteq M(x^2) \cap M(y^2)$ and $M([x,y]) \subseteq M(x^2) \cap M(y^2)$. \\In the specific case when $[x,y]=\be$ we have $M((x\ccolon y))=M(x^2) \cap M(y^2)$ and $M([x,y])=\varnothing$.
\item if $[x,y]=\be$ then $\w((xy)^2)=\w(x^2y^2)=\w(x^2)+\w(y^2)-2\w((x\ccolon y))$.
\end{enumerate}
\end{lemm}
\begin{proof} These items follow straightforwardly from Tables~\ref{table:1} and \ref{table:2}.
\end{proof}

\begin{lemm}\label{lem:r4b_p1}~
Let $\cC$ be a subgroup of $\cQ_8^{k_3}$ such that $\Phi(\cC)$ is a Hadamard code and let $a,b \in T(\cC)$. If $a,b,ab \not \in \{\be,\bu\}$, then $|M(a) \cap M(b)|=|M(a) \cap \overline{M(b)}|=|\overline{M(a)} \cap M(b)|=|\overline{M(a)} \cap \overline{M(b)}|=k_3/4$.
\end{lemm}
\begin{proof}
Straightforward.
\end{proof}

\begin{lemm}\label{lem:3}~
Let $\cC$ be a subgroup of $\Z_2^{k_1}\times \Z_4^{k_2}\times \cQ_8^{k_3}$ such that $\Phi(\cC)$ is a Hadamard code. Let $a,b,c\in \cC\setminus T(\cC)$.
\begin{enumerate}
\item\label{lem:3a} either $a^2=\bu$ or $[a,b]=[b,a]=\be$ or $[a,b]=[b,a]=a^ 2$.
\item\label{lem:3b} if $a^2 = u$ and $b^2 = c^2 = [b,c] \not \in \{\be,\bu\}$ then $[a,b] = \be$ or $[a,c] = \be$ or $[a,bc] = \be$.
\item\label{lem:3c} if $b^2=c^2=[b,c]$ and $[a,b]=[a,c]=\be$ then $(ab)^2=(ac)^2=\bu$ and $a^2$, $b^2$, $c^2$ are not equal to $\bu$.
\end{enumerate}
\end{lemm}
\begin{proof}
The first item was already proven in~\cite[Lemma IV.6]{rio-2013}.

For the second item we will assume that the first two possibilities of the conclusion are false. Using the first item in this Lemma we have $[a,b]=[a,c]=b^2=c^2$, so $[a,bc]=[a,b][a,c]=b^2c^2=e$. This proves the second item.

For the third item note that $(bc)^2=b^2c^2[b,c]=b^2=c^2$, thus $M(b^2)=M(c^2)=M((bc)^2)$.
Taken into account that $[a,b]=[a,c]=[a,bc]=\be$, by Lemma~\ref{lem:2} we have $M((a\ccolon b))=M((a\ccolon c))=M((a\ccolon bc))$. Hence, $(a\ccolon b)=(a\ccolon c)=(a\ccolon bc)$. Moreover, $(a\ccolon bc)=(a\ccolon b)(a\ccolon c)=(a\ccolon b)^2=\be$ and so $(a\ccolon b)=(a\ccolon c)=\be$. Now, using again Lemma~\ref{lem:2}, $\w(a^2b^2)=\w(a^2)+\w(b^2)-2\w((a\ccolon b))=\w(a^2)+\w(b^2)$. As we are working with elements of a Hadamard code, the weights must be equal to $n$, $n/2$ or $0$. The last possibility has been discarded when we state that they do not belong to $T(\cC)$, and so the only remainder possibility is $\w(a^2)=\w(b^2)=\w(c^2)=n/2$ and $\w(a^2b^2)=n$, proving in this way that $a^2$, $b^2$, $c^2$ are not equal to $\bu$ and $a^2b^2=\bu$. The same argumentation leads to $a^2c^2=\bu$.
\end{proof}

\section{The standard form for the generator set of a Hadamard $\qq$-code}\label{sec:standard}

To know the shape of a given Hadamard $\qq$-code we need to begin with a normalized generator set of $\cC$. Now, in this section we present a new point of view which lead to us to construct a standard generator set which will allow to decide the classification of a given subgroup in a more efficient way.

The next theorem shows that a subgroup $\cC$, such that $\phi(\cC)$ is a Hadamard $\qq$-code, has an abelian maximal subgroup $\cA $ which is normal in $\cC$ and $\cC / \cA$ is an abelian group of order $2^a$, for $a\in \{0,1,2\}$. We begin by a technical lemma.

\begin{lemm}\label{lemm:normal}
Let $\cC$ be a subgroup of $\Z_2^{k_1}\times \Z_4^{k_2}\times \cQ_8^{k_3}$ such that $\phi(\cC)=C$ is a Hadamard $\qq$-code. Let $\cA $ be any subgroup of $\cC$ containing $T(\cC)$, the subgroup of the elements of order two in $\cC$. Then $\cA $ is normal in $\cC$.
\end{lemm}
\begin{proof}
We want to show that $c^{-1}ac \in \cA$ for every $a \in \cA, c \in \cC$. We have $c^{-1}ac = a[a,c]$ and all commutators belong to $T(\cC)\subseteq \cA$, so the statement follows.
\end{proof}
\begin{theo}\label{main}
Let $\cC$ be a subgroup of $\Z_2^{k_1}\times \Z_4^{k_2}\times \cQ_8^{k_3}$ such that $\phi(\cC)=C$ is a Hadamard $\qq$-code. Then $\cC$ has an abelian maximal subgroup $\cA $ which is normal in $\cC$ and $|\,\cC / \cA \,|\in\{1,2,4\}$. Futher, $\cC$ may be expressed as a quotient of a semidirect product of $\cA$.
\end{theo}
\begin{proof}
A normalized generator set in \cite{rio-2013} has the form $\cC=\gen{x_1,\dots,x_{\sigma};$ $ y_1, \dots,y_{\delta};$ $z_1,$ $\dots,z_{\rho}}$, where $x_i$ are elements of order two that generates $T(\cC)=\gen{x_1...x_\sigma}$ and $Z(\cC)=\gen{x_1,\dots,x_{\sigma}$; $y_1,\dots,y_{\delta}}$ is the center of $\cC$.
Throughout this proof we will use a new generator set for $\cC$, which will be called \textit{standardized generator set}: $\cC=\gen{x_1,\dots,x_{\sigma}, r_1, \ldots, r_\tau, s_1,\ldots, s_\upsilon}$ and we always define the subgroup $\cA=\langle  x_1,\dots,x_{\sigma}, r_1, \ldots, r_\tau \rangle$, which is normal in $\cC$ by Lemma~\ref{lemm:normal}.

For the case when $\cC$ is of shape 1 or shape $1^*$ we have that the whole group $\cC$ is abelian, so $\cA=\cC$ and $|\,\cC / \cA\,|=1$.

For the case when $\cC$ is of shape 2 we have \cite{rio-2013} $\delta=0$, $z_1^2=z_2^2=[z_1,z_2]=\bu$, $[z_i,z_j]=z_j^2$ and $[z_j,z_k]=\be$ for every $i\in\{1,2\}$ and $3\le j,k \le \rho$.
We define the standardized generator set taking $x_1,\dots,x_{\sigma}$; $r_1 = z_1z_2$, $r_i=z_{i+1}$ for every $2 \le i \le \tau$; $s_1 = z_1$. Now we want to show that $\cA$ is abelian and maximal in $\cC$ and $\cC / \cA = \langle s_1 \rangle$.
Indeed, for every $2 \le i,j \le \tau$, $[r_1,r_i]=[z_1z_2,z_{i+1}]=[z_1,z_{i+1}][z_2,z_{i+1}]=z_{i+1}^2z_{i+1}^2=\be$ and  $[r_i,r_j]=[z_{i+1},z_{j+1}]=\be$. Hence $\cA$ is abelian.
To prove the maximality of $\cA$ in $\cC$ we show that $[s_1,r_1]=[z_1,z_1z_2]=[z_1,z_2]=\bu \neq \be$.
In addition, we see that $r_1^2=(z_1z_2)^2=z_1^2z_2^2[z_1,z_2]=\bu$ and $s_1^2=z_1^2=\bu$ and so $\cC=\cA \rtimes \langle s_1\rangle \big/ (\bu,s_1^2)$, with $r_1^2=\bu$.

For the case when $\cC$ is of shape 3 we have \cite{rio-2013} $\delta=0$, $z_1^2=\bu\not\in \gen{z_2^2,\dots,z_{\rho}^2}$, $[z_1,z_i]=z_i^2$ and $[z_i,z_j]=\be$, for every $i\not=j$ in $\{2,\dots,\rho\}$. We define the standardized generator set taking $r_i=z_{i+1}$ for every $1 \le i \le \tau=\rho-1$; $s_1 = z_1$.  Now we want to show that $\cA$ is abelian and maximal in $\cC$ and $\cC / \cA = \langle s_1 \rangle$.
Indeed, for every $1 \le i,j \le \tau$, $[r_i,r_j]=[z_{i+1},z_{j+1}]=\be$. Hence $\cA$ is abelian.
To prove the maximality of $\cA$ in $\cC$ we show that $[s_1,r_1]=[z_1,z_2]=z_2^2 \neq \be$.
In addition, we note that $\bu \not \in \gen{r_1^2 \dots r_\tau^2}$ and $s_1^2=z_1^2=\bu$ and so $\cC=\cA \rtimes \langle s_1\rangle \big/ (\bu,s_1^2)$, with $r_1^2\not=\bu$.

For the case when $\cC$ is of shape 4 with $\delta=0$ we have $\delta = 0$, $\rho=2$ and $z_1^2=z_2^2=[z_1,z_2] \not \in \{\be,\bu\}$. We define the standardized generator set taking $r_1=z_1$, $s_1=z_2$ and define $\cA = \gen{x_1,\dots,x_{\sigma}; r_1}$.
In this case, $r_1\not=\bu$ and $\upsilon=1$.
As all generators, except one, belong to $T(\cC)$ it is immediate that $\cA$ is abelian and $\cC / \cA = \gen{s_1}$.
For the maximality, see that $[r_1,s_1]=[z_1,z_2]=z_1^2 \neq \be$.
Note $r_1^2=s_1^2 \neq \bu$ and $\cC=\cA \rtimes \langle s_1\rangle \big/ (r_1^2,s_1^2)$.

For the case when $\cC$ is of shape $4^*$ we have $\rho=2$ and $z_1^2=z_2^2=[z_1,z_2] \not \in \{\be,\bu\}$. The element $y_1$ commutes with both $z_1, z_2$ and so, by item 3 of
Lemma \ref{lem:3} we have $y_1^2\neq\bu$ and $(y_1z_1)^2=(y_1z_2)^2=\bu$.
We define the standardized generator set taking $r_1 = y_1z_1$, $r_2=z_1$, $s_1=z_2$.
In this case, $r_1= \bu$ and $\upsilon=1$.
We have $[r_1,r_2]=[y_1z_1,z_1]=\be^2=\be$ and so $\cA$ is abelian. For the maximality, see that $[r_1,s_1]=[y_1z_1,z_2]=[z_1,z_2]=z_1^2 \neq \be$.
In addition, $r_1^2=(y_1z_1)^2=\bu \neq r_2^2=z_1^2$ and $s_1^2=z_2^2\neq \bu$ and $\cC=\cA \rtimes \langle s_1\rangle \big/ (r_1^2,s_1^2)$.

For the case when $\cC$ is of shape 5 we have $\delta=0$ and $\rho=4$. We have: $z_1^2=z_2^2=[z_1,z_2]=\bu\ne z_3^2=z_4^2=[z_3,z_4]$ and $[z_i,z_j] \in \langle z_j^2\rangle$ for every $i\in \{1,2\}$ and $j\in\{3,4\}$. We define the standardized generator set taking $r_1 = z_1$, $r_2=f(z_1)$, $s_1=z_2$, $s_2=f(z_2)$, where: \\
$f(z)=\left\{\begin{matrix}
z_3 \: \mbox{ if }\: [z,z_3]=e,\\
z_4 \: \mbox{ if } \: [z,z_4]=e,\\
z_3z_4 \: \mbox{ otherwise.} \\
\end{matrix}\right.$ \\
From Lemma~\ref{lem:3} it is easy to check that in the following matrix
$$
\left(\begin{array}{ccc}
{ }[z_1,z_3] & [z_1,z_4] & [z_1,z_3z_4] \\
{ }[z_2,z_3] & [z_2,z_4] & [z_2,z_3z_4]\\
{ }[z_1z_2,z_3] & [z_1z_2,z_4] & [z_1z_2,z_3z_4]\\
\end{array}\right)
$$
there is one and only one element in each row or column equal to $\be$, being the other two elements equals to $z_3^2=z_4^2$. Therefore, $[z_1,f(z_1)]=[z_2,f(z_2)]=\be$ and $[z_1,f(z_2)]=[z_2,f(z_1)]=[f(z_1),f(z_2)]=z_3^2=z_4^2$.

We have $\cA=\langle r_1,r_2 \rangle$ and $\cC / \cA =\langle s_1,s_2\rangle$. In particular $[r_1,r_2]=[z_1,f(z_1)]=\be$, hence $\cA$ is abelian.
For the maximality, see that $[r_1,s_1] = [z_1,z_2] \neq \be$ and $[r_2,s_2]=[f(z_1),f(z_2)] \neq \be$.
In addition, note $r_1^2=s_1^2=\bu\neq r_2^2=s_2^2$ and $\cC=\cA\rtimes \langle s_1, s_2\rangle\big/ (r_1^2,s_1^2)(r_2^2,s_2^2)$.
\end{proof}

The next corollary summarize the most relevant properties of the standardized set of generators we just defined.

\begin{coro}
Let $\cC$ be a subgroup of $\Z_2^{k_1}\times \Z_4^{k_2}\times \cQ_8^{k_3}$ such that $C=\Phi(\cC)$ is a Hadamard code and let \{\genset\} be a standard set of generators of $\cC$.
\begin{itemize}
\item The elements $x_i$ are of order two and generate $T(\cC)$.
\item The elements $r_i$ are of order four and commute with each other, $[r_i,r_j]=\be$ for every $1 \le i,j \le \tau$. When $\bu \in \gen{r_1 \dots r_\tau}$ we will take $\bu=r_1^2$ and we have $r_1^2=\bu \not \in \gen{r_2^2...r_\tau^2}$.
\item The cardinal $\upsilon$ of the set $\{s_1,s_\upsilon\}$ is in $\{0,1,2\}$ and when $\upsilon=2$ we have $s_1^2=\bu \neq s_2^2$, and $[s_1, s_2]=\be$. Moreover, when $r_1^2=s_1^2=\bu$ then $[r_1,s_1]=\bu$.
\item Any element $c\in \cC$ can be written in a unique way as
$$
c=\prod_{i=1}^{\sigma}x_i^{a_i}\prod_{j=1}^{\tau}r_j^{b_j}
\prod_{k=1}^{\upsilon}s_{k}^{c_k}, \,\,\,\mbox{where $a_i,b_j,c_k\in \{0,1\}$}.
$$
\end{itemize}
\end{coro}

There are a few more facts that we want to emphasize and that we use later in the next section about constructions. First of all, if there exists $x\in \cC$ such that $x^2=\bu$ then $k_1=0$. This is the reason why the shapes 2, 3, $4^*$ and shape 5 must have $k_1=0$. Secondly, if $x,y\in \cC$ and $[x,y]=z\not=\be$ then the components in $M(z)$ does not correspond to $\Z_4$. This is the reason why the shape 2 and 5 (where $r_1^2=s_1^2=[r_1,s_1]=\bu$) have $k_2=0$.

For shape 4  we have that $x^2\not=\bu$, for all $x\in \cC$. Hence, since $[r_1,s_1]=r_1^2=s_1^2\not=\bu$, we have that all $\Z_4$-components in $r_1$ and in $s_1$ are 0 or 2. The other generator elements in $\cC$ (apart from $r_1$ and $s_1$) are of order two. We claim that in this case $k_2=0$. Otherwise, after the Gray map, the corresponding Hadamard matrix would have repeated columns, those corresponding the two binary components of each $\Z_4$, and this contradicts that we had a Hadamard matrix. We can go further, for the same reason it can not exist any $Q_8$-component of order two in $r_1$ or $s_1$ and, since $[r_1,s_1]$ has weight $4k_3$ which must be $n/2$. We conclude that $k_1=4k_3$.

Finally, for shape 3 and shape $4^*$, the elements $x\in \cC$ such that $x^2=\bu$ should have their $\Z_4$-components belonging to $\{1,3\}$. The rest of elements $x$, so with $x^2\not=\bu$, have their $\Z_4$-components belonging to $\{0,2\}$  since, for any $i\in \{1,\ldots,\tau\}$, $[r_i,s_1]=r_i^2\not=\bu$ for shape 3 and $[r_i,s_1]=s_1^2\not=\bu$ for shape $4^*$.  The same argumentation as in the previous paragraph lead us to say that $k_2=2k_3$ for shape $4^*$.

The next theorem characterizes the maximal abelian subgroup $\cA$ and makes possible all constructions of these kind of Hadamard $\qq$-codes.
\begin{theo}\label{prop:z2z4}
Let $\cC$ be a subgroup of $\Z_2^{k_1}\times \Z_4^{k_2}\times \cQ_8^{k_3}$ such that $\phi(\cC)=C$ is a Hadamard $\qq$-code and $\cA $ the abelian maximal subgroup in $\cC$. Then $\phi(\cA)$ can be described as a duplication of a Hadamard $\Z_2\Z_4$-linear code when $\upsilon=1$ or as a quadruplication of a Hadamard $\Z_4$-linear code, if $\upsilon=2$.
\end{theo}
\begin{proof}
Let $\cA=\langle  x_1,\dots,x_{\sigma}, r_1, \ldots, r_\tau \rangle$ and $\cC=\langle \cA,s_1,\ldots,s_\upsilon \rangle$.
We know from Theorem~\ref{main} that $|\cC /\cA| \in \{1,2,4\}$.
If $|\cC /\cA| =1$ then there is nothing to prove, $\phi(\cA)$ is a Hadamard $\Z_2\Z_4$-linear code.

Let us assume that $\upsilon=1$, so $|\cC/\cA|=2$.
Code $\cA$ is additive and has length $2^m=2^{\sigma+\tau+\upsilon-1}$ .
Let $M$ be the matrix with the rows given by $x_1,\dots,x_{\sigma}, r_1, \ldots, r_\tau$. Matrix $M$ has $k_1$ binary columns, $k_2$ quaternary columns and $k_3$ quaternionic columns and is a generator matrix for $\cA$.
Let $\overline{M}$ the matrix $M$ extended with one more row given by $s_1$.
Also, let $N$ be the matrix where the rows are all elements in $\cA$ and $\overline{N}$ the matrix $N$ extended adding the remainder rows of $\cC$.
Columns in $N$ could be considered as binary columns after a Gray map of the original elements. First of all we claim that there are not three repeated binary columns in $M$.
Deny the claim. Let $a,b,c$ be the repeated binary columns in $N$ and $a', b', c'$ the corresponding extension to the second part of matrix $\overline{N}$. Since $\phi(\cC)$ is a Hadamard code any two columns of $\overline{N}$ agree in precisely half of components.
Hence, we should have $b'^{(c)}=a'$, where $b'^{(c)}$ means the complementary of $b'$. Also we should have $c'^{(c)}=a'$ and $c'^{(c)}=b'$ obtaining $b'^{(c)}=b'$ which could not happen. So, in $N$ there are not three or more repeated binary columns.
It is known that the parity check matrix $M$ of and additive code $\cA=\langle  x_1,\dots,x_{\sigma}, r_1, \ldots, r_\tau \rangle$ has at most $k_1+k_2$ different (up to sign) $\Z_2\Z_4$-columns~\cite{prv}, where $k_1+2k_2=2^{\sigma+\tau-1}$, and this maximum corresponds to a Hadamard code.
Therefore, since matrix $M$ has not three repeated binary columns and has length $2(k_1+2k_2)$ we conclude that $M$ has exactly $k_1+k_2$ different $\Z_2\Z_4-$columns (up to sign) each one repeated twice and $\cA$ is a duplicated additive Hadamard code.

Finally, let us assume $\upsilon=2$, so $|\cC/\cA|=4$. In this case we known that $k_1=k_2=0$ (Table~\ref{table:ex}). Matrix $M$ has repeated binary columns. Indeed, each element is in $\langle \ba \rangle \subset Q_8$ and so, each binary column is repeated twice.
We claim that we can not have five repeated binary columns, so three $Q_8$-columns.
Deny the claim. Let $a_1,a_2,a_3,a_4,a_5$ the five repeated binary columns in $N$ and $a'_1,a'_2,a'_3,a'_4,a'_5$ the corresponding extension to the second part of matrix $\overline{N}$. We know that in a Hadamard matrix, any three columns agree in precisely a fourth part of the components. Since any three different columns $a_{i_1}, a_{i_2}, a_{i_3}$ coincide we have that any three different columns $a'_{i_1}, a'_{i_2}, a'_{i_3}$ does not agree, simultaneously, in any component. This could not happen if we have five or more repeated columns $a_1,a_2,a_3,a_4,a_5$.
Therefore, since matrix $M$ has length $2^{\sigma+\tau}$ and has no five repeated binary columns, so $M$ has no three repeated $Q_8$-columns, we conclude that $M$ has duplicated all $Q_8$-columns. Hence, $M$ is a quatriplication of a $\Z_4$-linear hadamard code.
\end{proof}

From Theorem~\ref{prop:z2z4} and the rest of results of this section we see that there are two big classes of Hadamard $\qq$-codes. All codes contains the all one vector $\bu$, but there are codes where there exist an element $r_1$, of order four, such that $r_1^2 =\bu$ (codes of shape $1^*$, $2$, $4^*$ and $5$) and there are codes where $\bu$ is not the square of any other element of order four (codes of shape $1$, $3$ and $4$). We will define the new parameter $\tauB=\tau-1$ in the first case ($r_1^2=\bu$) and $\tauB=\tau$ in the second case ($r_1^2\not= \bu$). The existence conditions for Hadamard $\qq$-codes easily come from  Theorem~\ref{prop:z2z4} and \cite{prv}, where it was stated the existence conditions for Hadamard $\Z_2\Z_4$-linear codes.

Table~\ref{table:ex} summarizes what we have done in this section.

\begin{table}[ht]
$$
\begin{array}{|l|c|c|c|c|l|}
\hline
& \multicolumn{3}{c|}{\Z_2^{k_1}\times\Z_4^{k_2}\times\cQ_8^{k_3}}&&\\
\cline{2-4} \multicolumn{1}{|c|}{\mbox{shape}} &\multicolumn{1}{|c|}{k_1}&\multicolumn{1}{|c|}{k_2}&\multicolumn{1}{|c|}{k_3} &\multicolumn{1}{|c|}{\cC} &\multicolumn{1}{|c|}{\mbox{existence}}\\
\hline
\hline
1^* ~ (r_1^ 2=\bu,\tauB=\tau-1,\upsilon=0)  & 0 & 2^{\sigma+\tau-2} & 0 & \cA& \forall \tau \leq \lfloor \frac{m+1}{2}\rfloor;\\
&&&&&\sigma=m-\tau+1\\
\hline
1 ~~ (r_1^ 2 \ne \bu, \tauB=\tau,\upsilon=0)& 2^{\sigma-1}  & (2^\tau-1)2^{\sigma-2} & 0 &\cA& \forall \tau \leq \lfloor \frac{m}{2}\rfloor;\\
&&&&&\sigma=m-\tau+1 \\
\hline
2 ~~ (r_1^ 2=\bu=s_1^2,\tauB=\tau-1,\upsilon=1) & 0 & 0 & 2^{\sigma+\tau-2} & \cA \rtimes \Z_4 \big/ (\bu,s_1^2)& \forall \tau \leq \lfloor \frac{m}{2}\rfloor;\\
&&&&&\sigma=m-\tau\\
\hline
3 ~~ (r_1^ 2 \ne \bu=s_1^2, \tauB=\tau,\upsilon=1)  & 0 & 2^{\sigma-1} & (2^\tau-1)2^{\sigma-2} &\cA \rtimes \Z_4 \big/ (\bu,s_1^2)& \forall \tau \leq \lfloor \frac{m-1}{2}\rfloor;\\
&&&&&\sigma=m-\tau\\
\hline
4  ~~ (r_1^ 2 \ne \bu \ne s_1^2, \tauB=\tau,\upsilon=1) & 2^{\sigma-1} & 0 & 2^{\sigma-3} &\cA \rtimes \Z_4 \big/ (r_1^2,s_1^2) &\mbox{$m$ even; } \tau=1;\\
&&&&&\sigma = \frac{m}{2}+1\\
\hline
4^*  ~ (r_1^ 2 = \bu\ne s_1^2,\tauB=\tau-1,\upsilon=1) & 0 & 2^\sigma  & 2^{\sigma-1} &\cA \rtimes \Z_4 \big/ (r_2,s_1^2)& \mbox{$m$ even; } \tau=2;\\
&&&&&\sigma = \frac{m}{2}-1\\
\hline
5 ~~ (r_1^ 2=\bu,\tauB=\tau-1,\upsilon=2) & 0 &0& 2^{\sigma+1} &\cA\rtimes (\Z_4 \times\Z_4)\big/ (r_1^2,s_1^2)(r_2^2,s_2^2)& \tau=2;\\
&&&&&\sigma=m-3\\
\hline
\end{array}
$$
\label{table:ex}
\caption{Existence conditions and parameters $k_1,k_2,k_3$ depending on the shape of Hadamard $\qq$-codes of length $n=2^m$, where $m=\sigma+\tau+\upsilon-1$}
\end{table}

\section{Rank and kernel dimension of Hadamard $\qq$-codes}\label{sec:kernelrank}

In this section we show the conditions that $s_1$ and $s_2$ must fulfill in order to compose a code with specific values for the dimension of the kernel and the rank. Later, in the next section, our focus will be the construction of codes $\cC$, by adding the generators $s_1$ and, optionally, $s_2$ to a previous subgroup $\cA(\cC)$.

In the proof of Theorem~\ref{main} we saw that $r_1^2=\bu$ for codes of shapes $1^*, 2, 4^*,5$ and $\bu \notin \gen{r_1^2 \dots r_\tau^2}$ for the other shapes. \\Let $\cA(\cC)=\langle  x_1,\dots,x_{\sigma}, r_1, \ldots, r_\tau \rangle$ and let $\cR(\cC)$ be
defined by
$$\left\{\begin{array}{ll}
\cR(\cC)=\gen{x_1...x_\sigma,r_2...r_\tau};& \mbox{if $r_1^2=\bu$} \\
\cR(\cC)=\cA(\cC) ;& \mbox{if $r_1^2 \not= \bu$} \\
\end{array}\right.
$$


With this definition, we have the following technical lemmas.

\begin{lemm}\label{lem:10}
Let $a,b\in \cR(\cC) \backslash T(\cC)$ which are not in the same coset of $T(\cC)$, so $ba^{-1}\notin T(\cC)$ then:
\begin{enumerate}
\item $a^2, b^2, (ab)^2 \notin$ $\{\be,\bu\}$ and $\w(a^2) =\w(b^2) =\w((ab)^2)=n/2$
\item $\w((a\ccolon b))=n/4$ and so $(a\ccolon b) \notin \cC$.
\item With the same hypothesis as for $a,b$, let $a',b'$ a different pair, such that the different elements in $\{a,b,a',b'\}$ are pairwise not in the same coset of $T(\cC)$. Then $(a\ccolon b)\not=(a'\ccolon b')$.
\end{enumerate}
\end{lemm}
\begin{proof}~

\begin{itemize}
\item Elements $a,b$ are not in $T(\cC)$ so their square is not $\be$. Also, the construction of $\cR(\cC)$ explicitly excludes any element with square equal to $\bu$.
The product $ab$ is also an element of $\cR(\cC))$, thus their square can not be $\bu$. Moreover, if $(ab)^2=\be$ then $a=bT(\cC)$ which contradicts the hypothesis. This proves the first item.
\item As the elements $a,b$ commute, we have from Lemma~\ref{lem:2} $n/2=\w((ab)^2)=\w(a^2b^2)=\w(a^2)+\w(b^2)-2\w( (a\ccolon b) )=n/2+n/2-2\w((a\ccolon b))$. Hence, $\w((a\ccolon b)) = n/4$. This proves the second item.
\item Suppose $(a\ccolon b)=(a'\ccolon b')$. Since $\w((a\ccolon b))=\w((a'\ccolon b'))=n/4$ there are some components (for a total weight of $n/8$) where all $a,b,a',b'$ share an entry of order four. The rest of components of order four (for a total weight of $n/8$) in each $a,b,a',b'$ is not shared at all, since the elements are pairwise not in the same coset of $T(\cC)$. This situation is not possible in the case where all $a,b,a',b'$ are different, for we obtain a vector of length $5n/4$. If, without loss of generality, we suppose $b=a'$ we obtain $a^2b^2b'^2= \bu\in \cR(\cC)$, a contradiction.
\end{itemize}
\end{proof}

\begin{lemm}\label{lem:r4b_p2}
Let $\cC$ be a Hadamard $\qq$-code of shape 2 and length $n$ with a standard set of generators $\cC=\gen{x_1 \dots x_\sigma;r_1 \dots r_\tau;s_1}$ where, by definition of shape 2, $r_1^2=s_1^2=[r_1,s_1]=\bu$. Let $r$ be and element of order four in $\cA(\cC)$ such that $r^2 \neq \bu$ and let $x\in T(\cC)$ be an element which is not the square of any other element of $\cC$. If $(s_1:r_1)=x$ then
\begin{enumerate}
\item $|M((s_1 \ccolon r))|=k_3/4$; $\w((s_1 \ccolon r))=n/4$ and $(s_1 \ccolon r) \not \in \cC$.
\item for any element $r \in \cR(\cC)$ we have:
\begin{enumerate}
\item $(s_1\colon r) \not= (a\colon b)$, for any $a,b \in \cR(\cC)\backslash T(\cC)$ such that $ba^{-1}\in T(\cC)$.
\item $(s_1\colon r) \not= (s_1\colon a)$, for any $a \in \cR(\cC)\backslash T(\cC)$ such that $s_1a^{-1}\in \langle T(\cC),r_1\rangle$.
\end{enumerate}
\end{enumerate}
\end{lemm}

\begin{proof}
Taken into account that $[r_1,r]=\be$ and $r_1^2 = \bu \neq r^2$, from Lemma~\ref{lem:2} we have $M((s_1 \ccolon r))=M((s_1 \ccolon r_1)) \cap M(r^2)=M(x) \cap M(r^2)$ and from Lemma \ref{lem:r4b_p1} $|M(x) \cap M(r^2)|=k_3/4$. This proves the first part of the first item. Since for codes of shape 2 all non-zero components of order two has binary length four, the second part of the first item comes. As all elements of $\cC$ must have weight in $\{n,n/2,0\}$, we conclude that $(s_1 \ccolon r) \not \in \cC$ for any $r$. The first item is done.

For the second item, let $r_2,r_3,r_4$ elements of $\cR(\cC)\backslash T(\cC)$, pairwise not in the same coset of $T(\cC)$. First assume $(s_1 \ccolon r_2)=(s_1 \ccolon r_3)$. In this case we have $(s_1 \ccolon r_2r_3)=\be \in \cC$ which contradicts Lemma~\ref{lem:10}.
Now assume $(s_1 \ccolon r_2)=(r_3 \ccolon s_1)$ then, from Lemma~\ref{lem:3}, $(s_1 \ccolon r_2r_3)=(s_1 \ccolon r_2)(s_1 \ccolon r_3)=r_3^2(r_3 \ccolon s_1)^2=r_3^2 \in \cC$, against the above item in this lemma.
Finally, assume $(s_1 \ccolon r_2)=(r_3 \ccolon r_4)$. Since $M((s_1 \ccolon r_2))=M((s_1 \ccolon r_1)) \cap M(r_2^2)=M(x) \cap M(r_2^2)$ and $M((r_3 \ccolon r_4))=M(r_3^2) \cap M(r_4^2)$, we obtain $M(x) \cap M(r_2^2) = M(r_3^2) \cap M(r_4^2)$. The above equality  means that $M(x)$, $M(r_2^2)$, $M(r_3^2)$, $M(r_4^2)$ have $k_3/4$ elements in common, while the other $k_3/4$ elements in each one of these sets are disjoint from each other. This is not possible, for a total of $5k_3/4$ components is needed for the above composition.
The second item is proved.
\end{proof}

Now, we can enumerate a list of cases depending on $\tau$, $\tauB$ and $\upsilon$ from which we obtain later, in the next section, a code with an specific dimension of the kernel and rank.

\begin{prop}\label{bounds}
Let $\cC$ be a subgroup of $\Z_2^{k_1}\times \Z_4^{k_2}\times \cQ_8^{k_3}$ such that $C=\Phi(\cC)$ is a Hadamard code generated by $\gen{\cA(\cC),s_1,s_\upsilon}$. The values of the dimension of the  kernel and rank depends on $\tau$, $\tauB$ and $\upsilon$ according to the following cases:

\begin{enumerate}
\item In the case $\upsilon=0$ ($\Z_2\Z_4$-code) we have
\begin{enumerate}
\item if $\tauB \le 1$ (that is, $\tau \le 2$ and $r_1^2=\bu$ or $\tau \le 1$ and $r_1^2 \neq \bu$) the code is linear and $k=r=\sigma+\tau$;
\item if $\tau \ge 3$ and $r_1^2=\bu$ then $k = \sigma + 1$, $r=\sigma+\tau+\binom{\tau-1}{2}$ and $C$ is a $\Z_2\Z_4$-linear code of type $2^{\sigma-\tau}4^{\tau}$;
\item if $\tau \ge 2$ and $r_1^2\not=\bu$ then $k = \sigma$, $r=\sigma+\tau+\binom{\tau}{2}$ and $C$ is a $\Z_4$-linear code of type $2^{\sigma-\tau}4^{\tau}$.
\end{enumerate}
\item In the case $\tau=1, \upsilon = 1$ we have
\begin{enumerate}
\item if $(s_1\ccolon r_1) \in \cC$ then $C$ is linear and $k=r=\sigma+2$;
\item if $(s_1\ccolon r_1) \notin \cC$ then $k=\sigma$ and $r=\sigma+3$.
\end{enumerate}
\item In the case $\tau=2, \tauB=1, \upsilon=1$, consider the swappers $(s_1\ccolon r_1)$, $(s_1\ccolon r_2)$ and $(s_1\ccolon r_1r_2)$:
\begin{enumerate}
\item if these three swappers are in $\cC$ then $k = r =  \sigma + 3$;
\item if one of them is in $\cC$ then $k = \sigma + 1$, $r=\sigma+4$;
\item if none of them is in $\cC$ then $k = \sigma, r=\sigma+5$.
\end{enumerate}
\item In the case $\tau \ge \tauB \ge 2, \upsilon=1$ we have
\begin{enumerate}
\item if $(bs_1\ccolon a) \in \cC$ for all $a \in \cA(\cC)$ and some $b\in \cR(\cC)$, then $k = \sigma + \tau-\tauB + 1$, $r=\sigma+\tau+1+\binom{\tauB}{2}$;
\item if the previous condition is not satisfied but $(r_1\ccolon s_1) \in \cC$ and $r_1^2=\bu$ then $k = \sigma+1$ and $r = \sigma+\tau+1+\binom{\tauB+1}{2}$;
\item if none of the previous conditions is satisfied then $k = \sigma$ and
\begin{eqnarray*}
\sigma+\tau+\upsilon+\binom{\tau-1}{2} \le r \le \sigma+\tau+\upsilon+\binom{\tau}{2}+1 \,\mbox{when } \tau=\tauB+1,(r_1^2=\bu)\\
\sigma+\tau+\upsilon+\binom{\tau}{2}+1 \le r \le \sigma+\tau+\upsilon+\binom{\tau+1}{2}, \,\mbox{when } \tau=\tauB, (r_1^2 \ne \bu)
\end{eqnarray*}
\end{enumerate}
\item In the case $\tau=2 \lland \upsilon=2$, consider the two swappers $(r_2\ccolon s_2)$ and $(r_1r_2\ccolon s_1s_2)$.
\begin{enumerate}
\item if both swappers are in $\cC$ then $C$ is linear;
\item if only one of the two swappers is in $\cC$ then $k = \sigma + 2$ and $r=\sigma+5$
\item if none of the two swappers is  in $\cC$ then $k = \sigma$ and $r=\sigma+6$
\end{enumerate}
\end{enumerate}
\end{prop}

\begin{proof}~

\begin{enumerate}
\item In the case $\upsilon=0$ we have that $\cC$ is abelian and the shape is 1.\\
If $r_1^2=\bu$, any element $c \in \cC$ can be written as $c=xr_1^{i}r$ with $x \in T(\cC)$, $r \in R(\cC)$ and $i \in \{0,1\}$. The element $\phi(c)$ belongs to $K(C)$ if the swapper of $c$ with every element in $\cC$ is still in $\cC$. Hence, $\phi(c)\in K(C)$ if and only if $(c\ccolon r_1) \in \cC$ and $(c\ccolon r) \in \cC$.
From Lemma~\ref{previs} we see that $(r\ccolon r_1)=(r\ccolon r)=r^2$ and for all $\rB \in R(\cC)$ with $\rB \not= rT(\cC)$ we have $(r\ccolon \rB) \notin \cC$. Hence, we conclude with $K(C)=\langle T(\cC),r_1 \rangle$. From Lemma \ref{lem:10}, all swappers of two elements in $R(\cC)$ are different. Therefore, if $\tauB > 1$ and $r_1^2=\bu$ we have $k=\sigma+1$ and $r=\sigma+\tau+\binom{\tau-1}{2}$. With the same argumentation, if $\tauB > 1$ and $r_1^2\not=\bu$ we have $k=\sigma$ and $r=\sigma+\tau+\binom{\tau}{2}$. In both cases, if $\tauB \le 1$ we have $r-k \leq 1$ and so the code is linear and $k=r=\sigma+\tau$.

\item The case $\tau=1 \lland \upsilon = 1$ means that $\cC=\langle T(\cC), r_1, s_1\rangle$. Since the linear span of $C$ is generated by $\cC$ and all swappers of pairs of elements in $\cC$ we conclude that $\langle C \rangle$ is the binary image by the Gray map of $\langle \cC, (r_1\ccolon s_1)\rangle$. Hence, either $(s_1\ccolon r_1) \in \cC$ and the code is linear or $k=\sigma$ and $r=\sigma+\tau+\upsilon+1=\sigma+3$.

\item In the case $\tau=2 \lland \upsilon=1 \lland \tauB = 1$ the code $\cC$ is of shape 2 with $r_1^2=s_1^2=\bu \neq r_2^2$ or is of shape $4^*$ with $r_1^2=\bu \neq s_1^2=r_2^2$.
Any element $c \in \cC$ can be written as $c=xr_1^{i}r_2^{j}s_1^{k}$ where $x \in T(\cC)$ and $i,j,k\in \{0,1\}$. It belongs to $K(\cC)$ if the swapper of $c$ with every element in $\cC$ is still in $\cC$. Hence, from Lemma~\ref{previs}, $c\in K(C)$ if and only if $(s_1^{k}\ccolon r_1),(s_1^{k}\ccolon r_2) \in \cC$ (recall that when $(s_1\ccolon r_1) \in \cC \lland (s_1\ccolon r_2) \in \cC$ then also $(s_1\ccolon r_1r_2) \in \cC$).
If all swappers above are in $\cC$ then code $C$ is linear and  $K(C)=\langle T(\cC),r_1,r_2,s_1 \rangle$. \\
If some swapper does not belong to $\cC$, for instance, $(s_1\ccolon r_1) \in \cC \lland (s_1\ccolon r_2) \not \in \cC$ then $(s_1\ccolon r_1r_2) =(s_1\ccolon r_1)(s_1\ccolon r_2)\notin \cC$, hence $K(C)=\langle T(\cC),r_1\rangle$ and $\langle C \rangle = \phi(\langle T(\cC), r_1, r_2, s_1, (s_1\ccolon r_2)\rangle)$. The same argumentation works for the other instances proving the statement.\\
If none of the swappers belong to $\cC$ then $\langle C \rangle = \phi(\langle T(\cC), r_1, r_2, s_1, (s_1\ccolon r_2), (s_1\ccolon r_1) \rangle)$ and $K(\cC)=T(\cC)$.

Note that if $s_1^2 \neq \bu$ then $M((r_1r_2)^2) \cap M(s_1^2) = \varnothing$, thus by Lemma \ref{lem:2}, $(r_1r_2\ccolon s_1)=\be$.

\item
In the case $\tau \ge \tauB \ge 2 \lland \upsilon=1$ the code is of shape 2 with $r_1^2=s_1^2=\bu$,  or shape 3 with $s_1^2=\bu \not \in \gen{r_1^2 \dots r_{\tau}^2}$.

Assume $r_1^2=\bu$ and $(bs_1\ccolon a) \in \cC$ for all $a \in \cA(\cC)$ and some $b \in \cR(\cC)$. We can assert that $(s_1\ccolon r_1) \in \cC$ because, from Lemma~\ref{previs}, $(bs_1\ccolon r_1)=(b\ccolon r_1)(s_1\ccolon r_1)=b^2(s_1\ccolon r_1) \in \cC$. In this way we have $K(C)=\gen{T(\cC),r_1,bs_1}$  and the linear span is generated by $\cC$ and the swappers of pairs in $\cR(\cC)$ (from Lemma-\ref{lem:10} we known there are a total amount of $\binom{\tauB}{2}$ swappers of this kind to be included in the generator set of the linear span). If $r_1^2\not=\bu$  we have $K(\cC)=\gen{T(\cC),bs_1}$ and the linear span is generated as before. Hence, $k = \sigma +1+\tau-\tauB$, $r=\sigma+\tau+\upsilon+\binom{\tauB}{2}$.

Assume now $r_1^2=\bu \lland (s_1\ccolon r_1) \in \cC$ (but not the previous condition about $(bs_1\colon a)$). If $(s_1\ccolon r_1)=b^2$ where $b$ is an element of $\cA(\cC)$, then $(bs_1 \ccolon r_1)=(b \ccolon r_1)(s_1 \ccolon r_1)=b^2b^2=\be$.
So $(bs_1 \ccolon a)=\be$ for all $a \in \cA(\cC)$ and some $b \in \cR(\cC)$, and the previous condition is fulfilled, a contradiction. Thus, without lost of generality, we can assert $(s_1\ccolon r_1)=x$ with $x \in T(\cC)$ and $x$ is not the square of any other element of $\cC$.
In this case $K(\cC)=\gen{T(\cC),r_1}$ and the linear span of $C$ is generated by $\cC$, the swappers of pairs in $\cR(\cC)$ (from Lemma-\ref{lem:10}, $\binom{\tauB}{2}$ swappers), and the swappers of $s_1$ with the elements in $R(\cC)$ (from Lemma~\ref{lem:r4b_p2}, $\tauB$ swappers). Hence, $k = \sigma+1$, $r = \sigma+\tau+\upsilon+\binom{\tauB}{2}+\tauB = \sigma+\tau+\upsilon+\binom{\tauB+1}{2}$.

When $(s_1\ccolon r_1) \not \in \cC \lland r_1^2 =\bu$, we have $\tauB=\tau-1$ and $K(\cC)=T(\cC)$. The linear span is generated by $\cC$, the swappers of pairs in $\cR(\cC)$ (from Lemma-\ref{lem:10}, $\binom{\tauB}{2}$ swappers), the swapper $(r_1\ccolon s_1)$ and the swappers of $s_1$ with elements of $\cR(\cC)$ that do not belong to $\cC$. Hence $k = \sigma$, $\sigma+\tau+\upsilon+\binom{\tauB}{2} \le r \le \sigma+\tau+\upsilon+\binom{\tauB}{2}+\tau=\sigma+\tau+\upsilon+\binom{\tauB+1}{2}+1$.

When $(s_1\ccolon r_1) \not \in \cC \lland r_1^2 \not=\bu$, we have $\tau=\tauB$. Now $K(\cC)=T(\cC)$ and the linear span is generated by $\cC$, the swapper of elements in $\cR(\cC)$ and the swappers of $s_1$ with elements of $\cR(\cC)$ that do not belong to $\cC$ (at least one of these last swappers must not belong to $\cC$, otherwise we are in the previous case 4a). Hence, $\sigma+\tau+\upsilon+\binom{\tauB}{2}+1 \le r \le \sigma+\tau+\upsilon+\binom{\tauB}{2}+\tau=\sigma+\tau+\upsilon+\binom{\tauB+1}{2}$.

\item In the case $\tau=2 \lland \upsilon=2$ the code $\cC$ is of shape 5 with $r_1^2=s_1^2=\bu \neq r_2^2=s_2^2$, $\cC=\gen{T(\cC); r_1,r_2; s_1,s_2}$.

We have $\cC=\gen{T(\cC); r_1,r_2; s_1,s_2}=\gen{T(\cC); r_1r_2,r_2; s_1s_2,s_2}$), so it is enough to analyze the swappers $(r_1r_2\ccolon r_2)$, $(r_1r_2\ccolon s_1s_2)$, $(r_1r_2\ccolon s_2)$, $(r_2\ccolon s_1s_2)$, $(r_2\ccolon s_2)$ and $(s_1s_2\ccolon s_2)$.

From Lemma~\ref{previs}, $(r_1r_2\ccolon r_2)=(r_1\ccolon r_2)(r_2\ccolon r_2)=r_2^2r_2^2=\be$, and $(s_1s_2\ccolon s_2)=$$(s_1\ccolon s_2)(s_2\ccolon s_2)=s_2^2s_2^2=\be$.
Moreover, we have $(r_1r_2)^2=(s_1s_2)^2=\bu r_2^2=\bu s_2^2$, so $\Supp((r_1r_2)^2)$ is disjoint from $\Supp(s_2^2)$ and also from $\Supp(r_2^2)$. The same for $\Supp((s_1s_2)^2)$, which is disjoint from $\Supp(r_2^2)$ and $\Supp(s_2^2)$. Hence, $(r_1r_2\ccolon s_2)=(s_1s_2\ccolon r_2)=(r_1r_2\ccolon r_2)=(s_1s_2\ccolon s_2)=\be$.
Now, only swappers $(r_1r_2\ccolon s_1s_2)$ and $(r_2\ccolon s_2)$ are left in our analysis. When $(r_1r_2\ccolon s_1s_2) \in \cC$ and $(r_2\ccolon s_2) \in \cC$ we have the linear case $K(C)=S(C)=\cC$; if only one of these swappers belongs to $\cC$ we have the case where $k=\sigma+2$ and the linear span is generated by $\cC$ and the swapper that does not belong to $\cC$. Finally, if none of these two swappers belongs to $\cC$ we have $K(\cC)=T(\cC)$ and the linear span is generated by $\cC$ and the two swappers.
\end{enumerate}
\end{proof}

From the above results in Proposition~\ref{bounds} we can specify a little more for what shapes we obtain the values for the rank and dimension of the kernel.
Tables~\ref{taula1} and \ref{taula2} are a summary. We write in the right top corner of each value the corresponding shape and, separate by a colon, the item of Proposition~\ref{bounds} where the case is studied. Note that we do not include the additive cases which were studied in~\cite{rio-2013}. Table~\ref{taula1} covers all items in  Proposition~\ref{bounds}, except 4). Table~\ref{taula2} covers the codes in item 4) of Proposition~\ref{bounds}.
\begin{table}\label{taula1}
$$
\begin{array}{c||c|c|c}
m+1-k& \multicolumn{3}{c}{r-(m+1)} \\
\hline\hline
4 &-&- &  2 ^{[5;5c]} \\
3 &- & 2 ^{[2,4^*;3c]} & - \\
2 &1 ^{[2,3,4;2b]} &1 ^{[2,4^*;3b]} &1^{[5;5b]} \\
0 & 0 ^{[2,4^*;2a]} & 0 ^{[2,4^*;3a]} & 0 ^{[5;5a]}
\end{array}
$$
\caption{Rank and dimension of the kernel for the codes fulfilling Proposition~\ref{bounds}, except item 4)}
\end{table}

\begin{table}\label{taula2}
$$
\begin{array}{c||c|c}
m+1-k & \multicolumn{2}{c}{r-(m+1)} \\
\hline\hline
\tau+1 &\binom{\tau-1}{2}\cdots\binom{\tau}{2}+1 ^{[2;4c]} & \binom{\tau}{2}+1\cdots\binom{\tau+1}{2} ^{[3;4c]}  \\
\tau & \binom{\tau}{2}^{[2;4b]} & \binom{\tau}{2} ^{[3;4a]}\\
\tau-1 & \binom{\tau-1}{2} ^{[2;4a]} &-
\end{array}
$$
\caption{Rank and dimension of the kernel for the codes fulfilling Proposition~\ref{bounds}, item 4)}
\end{table}

\section{Construction of Hadamard $\qq$-codes}\label{sec:construction}

In this section it is shown how to construct Hadamard $\qq$-codes with any allowable pair of values for the rank and the dimension of the kernel. We follow the entries of Proposition~\ref{bounds} explaining, in each case, how to construct the desired Hadamard code. After the constructions, as a summary, we include Theorem~\ref{fites}, where it is described what are the allowable parameters for the dimension of the kernel and, for each one of these values, it is said what is the range of values for the rank. For each one of the possible pair of allowable values for the dimension of the kernel and rank, we construct a Hadamard $\qq$-code fulfilling it. As an illustration of the constructions we include two examples at the end of the section.

We can take as starting point a Hadamard $\Z_2\Z_4$-code $\cD$ of type $2^{\sigma-\tau}4^{\tau}$, which can be constructed using the methods described in \cite{prv,Z2Z4}. Recall that an element $x$ with square equal to $\bu$ is included in $\cD$ if and only if $\cD$ is a $\Z_4$-code.

We define three basic homomorphisms:
\begin{equation}\label{morph1}
\begin{array}{lll}
\chi_1: & \Z_2 \rightarrow \Z_4 &\mbox{such that $\chi_1(x)=2x$,}\\
\chi_2: & \Z_4 \rightarrow \cQ_8 &\mbox{such that  $\chi_2(x)=\ba^x$,}\\
\chi_3: & A \rightarrow A \times A &\mbox{such that  $\chi_3(x)=(x,x)$, where $A\in\{\Z_2,\Z_4\}$.}
\end{array}
\end{equation}

The next theorem is one of the main results in the current paper. For any given pair of allowable parameters $r,k$ we construct a Hadamard $\qq$-code with these parameters.

\begin{theo}\label{constructions}
Let $C$ a Hadamard $\qq$-code of length $2^m$ and $|T(\cC)|=2^\sigma$, where $T(\cC)$ is the subgroup of elements in $\cC$ of order two; $| \cC / \cA(\cC)|=2^\upsilon$; $|\cA(\cC) / T(\cC)|=2^\tau$; $|\cC / T(\cC)|=2^{\tau+\upsilon}$ and $m+1=\sigma+\tau+\upsilon$. Then, for any two allowable values of the rank $r$ and dimension of the kernel $k$ (Proposition~\ref{bounds}) of a putative Hadamard $\qq$-code, we construct this code.
\end{theo}
\begin{proof}
First, we describe the elementary constructions based on the previously defined homomorphisms~(\ref{morph1}). After that, for each of all possibilities given in Lemma~\ref{bounds}, we show how to construct the putative code.

The first step is the construction of the subgroup $\cA(\cC)$ applying~(\ref{morph1}) to some Hadamard $\Z_2\Z_4$-code $\cD$, of type $2^{\sigma-\tau}4^{\tau}$, which can be constructed using the methods described in \cite{prv,Z2Z4}. After obtaining this code $\cD$, from Theorem~\ref{prop:z2z4} we will duplicated it (or quatruplicate it). One of the following ways must be used:
\begin{itemize}
\item From Table~\ref{table:ex} we know that codes of shape 1 or shape $1^*$ have $k_3=0$, so they are Hadamard $\Z_2\Z_4$-codes, $\cC=\cA(\cC)=\cD$.
\item In a code of shape 2, from Table~\ref{table:ex} we have that $k_1=k_2=0$. Also $r_1^2=\bu$ and $\cD$ is a $\Z_4$-linear Hadamard code. So $\cA(\cC)$ can be obtained applying $\chi_2$ component-wise to $\cD$.
\item In a code of shape 3, since $r_1^2 \ne \bu$, $\cD$ must be a $\Z_2\Z_4$-code. Moreover, from Table~\ref{table:ex}, $k_1=0$ and  $\cA(\cC)$ can be obtained applying $\chi_1$ component-wise to $\Z_2$ components of $\cD$ and $\chi_2$ to $\Z_4$ ones. This means that the rate between the number of $\Z_2$ and $\Z_4$ components in a Hadamard $\Z_2\Z_4$-code is extended to theses codes of shape 3.
\item In a code of shape 4, since $r_1^2 \ne \bu$ we have that $\cD$ must be a $\Z_2\Z_4$-code. As $\tau=1$, this code must be of type $2^{2k}4^k$. From Table~\ref{table:ex} $k_2=0$, so $\cA(\cC)$ can be obtained applying $\chi_3$ component-wise to $\Z_2$ components of $\cD$ and $\chi_2$ to the $\Z_4$ components.
\item In a code of shape $4^*$, since $r_1^2 = \bu$ we have that $\cD$ must be a $\Z_4$-code. $\cA(\cC)$ can be obtained applying $\chi_3$ component-wise to half of the $\Z_4$ components of $\cD$ and $\chi_2$ to the rest of components.
\item In a code of shape 5, since $r_1^2 = \bu$, $\cD$ must be a $\Z_4$-code. From Table~\ref{table:ex}, $k_1=k_2=0$, so $\cA(\cC)$ can be obtained applying $\chi_3$ component-wise to all $\Z_4$ components of $\cD$ followed by $\chi_2$ to obtain the desired $\cA(\cC)$.

\end{itemize}

The code $\cC$ is obtained adding one generator (two in case of shape 5) to $\cA(\cC)$. The components of these added generators must follow some restrictions to be sure that a Hadamard code is obtained. In addition, several values remains free of choice. These values are used  to select the target rank and dimension of the kernel.

Specifically, we follow the list of all possibilities given in Proposition~\ref{bounds} and, for each one of them, we show how to construct the putative code.

For codes in item 2) of Proposition~\ref{bounds} we can construct a code of shape 2 (it is also possible to construct codes of shape 3 or shape 4). Code $\cD$ must be in $\Z_4^{m-2}$ with $r_1^2 = \bu$, length $2^{m-1}$, $\sigma=m-1$ and $\tau=1$. This code is of type $2^{\sigma-1}4^1$~\cite{prv}. Now, we add one new generator $s_1$ to $\cA(\cC)$ with its components in $\{\bb,\ba^2\bb\}$ or $\{\ba\bb,\ba^3\bb\}$ obtaining a code $\cC$ with the following generator matrix.

$$\cC=\left( \begin{array}{ccc}
\cA(\cC)&=&\chi_2(\cD) \\
\hline
s_1^{(1)} & \dots & s_1^{(2^{m-2})}
\end{array}
\right)
$$

If we take all $s_1^{(i)}$ components in $\{\bb,\ba^2\bb\}$ then $(s_1 \ccolon r_1) \in \cC$ then $\cC$ is in the subcase 2a) of Proposition~\ref{bounds} where it is proven that the obtained code is linear. If we fill one of the components of $s_1$ with one value in $\{\bb,\ba^2\bb\}$ and the remainder ones with values in $\{\ba\bb,\ba^3\bb\}$ then $(r_1\ccolon s_1) \notin \cC$ (if $m > 3$), we obtain the subcase 2b) where it is proven that $k=\sigma$, $r=\sigma+3$.

For codes in item 3) of Proposition~\ref{bounds} the shape is 3. We begin by taking a Hadamard $\Z_4$-linear code $\cD$ with $r_1^2 = \bu$ and length $2^{m-1}$, $\tau=2$ and $\sigma=m-2$. This code is of type $2^{\sigma-2}4^2$ and the number of $\Z_4$ components is $2^{m-2}$~\cite{prv}.
If $s_1$ is constructed with all components in $\{\bb,\ba^2\bb\}$ then all swappers of $s_1$ with elements of $\cA(\cC)$ are in $\cC$ and we have a linear code according to sub-case 3a) of Proposition~\ref{bounds}.
If $s_1$ takes values in $\{\bb,\ba^2\bb\}$ for all components where $r_2$ has order four, plus one more component (which can be randomly selected), and we take the rest of the components of $s_1$ in $\{\ba\bb,\ba^3\bb\}$, then $(r_1\ccolon s_1) \notin \cC$, $(r_2 \ccolon s_1) \in \cC$ and we are fulfilling item 3b) of Proposition~\ref{bounds}, where it is proven that $k=\sigma+1$ and $r=\sigma+4$.
Finally, if we fill one of the components of $s_1$ with a value in $\{\bb,\ba^2\bb\}$ and the remainder components with values in $\{\ba\bb,\ba^3\bb\}$ then $(r_1\ccolon s_1) \notin \cC$, $(r_2 \ccolon s_1) \not \in \cC$, $(r_1r_2 \ccolon s_1) \not \in \cC$ and we obtain the subcase 3c) of Proposition~\ref{bounds}, where it is proven that $k=\sigma$, $r=\sigma+5$.

For codes in item 4) of Proposition~\ref{bounds} the shape is 2 or 3. We start by taking a Hadamard $\Z_4$-linear code $\cD$ with $r_1^2 = \bu$ and length $2^{m-1}$, $\sigma+\tau=m-2$. This code is of type $2^{\sigma-\tau}4^\tau$~\cite{prv}.
If we select all components of $s_1$ in $\{\bb,\ba^2\bb\}$ then $(s_1 \ccolon r_1)=\be$, we are in the case 4a), where it  is proven that $k = \sigma + \tau-\tauB + 1$, $r=\sigma+\tau+1+\binom{\tauB}{2}$. If we select an element $x \in T(\cC)$ which is not the square of any other element in $\cA(\cC)$ and we replace its zero components by $\bb$ and the components $\ba^2$ by $\ba\bb$, then $(s_1 \ccolon r_1)=x$ and so we are in the case 4b) where it is proven that $k = \sigma+1$ and $r = \sigma+\tau+1+\binom{\tauB+1}{2}$.
To reach the upper bound for the rank in the case 4c),  $k = \sigma$ and $r = \sigma+\tau+1+\binom{\tau}{2}+1$, we need the following construction.

Split all components in two sets taken into account if the value of $r_2$ either has order four or not.
Split each one of these two sets  according if the value of $r_3$ is either of order four or not. Repeat the process again and again for each $r_i$ until $r_\tau$. Since $\langle r_1^2, r_2^2, \ldots, r_{\tau}^2 \rangle$ is a linear subspace of the Hadamard code we obtain $2^{\tau-1}$ sets with $k_3/2^{\tau-1}$ components in each one.
Now, construct the element $s_1$ with the value $\bb$ in all components, except for one component in each one of the previous sets, where we put the value $\ba\bb$.
In this way $(s_1,r_i) \not \in \cC$ for any $2 \ge i \ge \tau$, which assure to obtain the maximum rank $\sigma+\tau+1+\binom{\tau-1}{2}+\tau= \sigma+\tau+1+\binom{\tau}{2}+1$.

A value of the rank equal to one less than the above maximum can be reached if in the constructed $s_1$ we put the value $\bb$ in all components where $r_2$ has order four. Repetitively, we can decrease by one the value of the previous rank by putting the value $\bb$ in all components where $r_2$ or $r_3$ has order four and so on. The lower rank we obtain is $\sigma+\tau+1+\binom{\tau-1}{2}+1$. To obtain the lower limit for the rank, $r = \sigma+\tau+1+\binom{\tau-1}{2}$, we take the value $\bb$ in a component of $s_1$ if some of the generators $r_2 \dots r_\tau$ has the respective component of order four. Otherwise the value $\ba\bb$. Example~\ref{ex:1} shows these constructions.

Now, we deal with codes in item 4, c) of Proposition~\ref{bounds}, of shape 3. The starting point is a Hadamard $\Z_2\Z_4$-linear code $\cD$ with $r_1^2 \not= \bu$ and length $2^{m-1}$, $\sigma+\tau=m-2$. This code is of type $2^{\sigma-\tau}4^\tau$ and has $2^{\sigma-1}$ binary components and $(2^\tau-1)2^{\sigma-2}$ quaternary components~\cite{prv}. After applying $\chi_1,\chi_2$ to the binary and quaternary components, respectively, we define $s_1$ taking in all the quaternary components de value $1$ and using the same technique as before, splitting the quaternionic components, to decide the values in these components.
We obtain $(2^\tau-1)$ sets with $2^{\sigma-2}$ components in each one. The maximum rank we obtain is $\sigma+\tau+1+\binom{\tau}{2}+\tau= \sigma+\tau+1+\binom{\tau+1}{2}$. The minimum is not as before, but $\sigma+\tau+\upsilon+\binom{\tau}{2}+1$. Indeed, when the rank is $\sigma+\tau+\upsilon+\binom{\tau}{2}$, the constructed $s_1$ belongs to the kernel and so $k=\sigma+1$ (this corresponds to the case 4a). Example~\ref{ex:2} shows these constructions.

For codes in item 5) of Proposition~\ref{bounds} we must construct a code of shape 5. We begin by taking a Hadamard $\Z_4$-linear code $\cD$ with $r_1^2 = \bu$, length $2^{m-2}$, $\tau=2$ and $\sigma=m-4$. This code is of type $2^{\sigma-3}4^2$ and the number of quaternary components is $2^{m-3}$~\cite{prv}. Now, as we said before, we can obtain $\cA(\cC)$ as $\chi_2(\chi_3(\cD))$ and $\cC=\langle \cA(\cC), s_1, s_2 \rangle$. The following matrix is a generator matrix for $C$:
$$\left( \begin{tabular}{cccccccccc}
\multicolumn{10}{c}{$\cA(\cC)=\chi_2(\chi_3(\cD))$} \\
\hline
$s_1^{(1)}$ & & \dots & & $s_1^{(2^{m-3})}$ & $1$ & $\ba^2$ & $\stackrel{2^{m-3}}{\dots}$ & $1$ & $\ba^2$ \\
$1$ & $\ba^2$ & $\stackrel{2^{m-3}}{\dots}$ & $1$ & $\ba^2$ & $s_2^{(2^{m-3}+1)}$ & & \dots & & $s_2^{(2^{m-2})}$ \\
\end{tabular}
\right)
$$
If all components of order four of $s_1$ and $s_2$ are in $\{\bb,\ba^2\bb\}$ then all swappers of $s_1$ and $s_2$ with elements of $\cA(\cC)$ are in $\cC$, so we have the linear case according to item 5a) of Proposition~\ref{bounds}.
Say that the first component of $r_1$ is of order four, but the first component of $r_2$ is of order at most two. Take $s_1^{(1)}=\ba\bb$ and the rest of components of order four of $s_1$ and $s_2$ are equal to $\bb$ then $(s_1s_2\colon r_1r_2) \not \in \cC$, $(s_2 \ccolon r_2) \in \cC$ and item 5b) of Proposition~\ref{bounds} is fulfilled, reaching a code with $k=\sigma+2$ and $r=\sigma+5$.
Finally, if $s_1^{(1)}=s_2^{(2^{m-3}+1)}=\ba\bb$ and the rest of components of order four of $s_1$ and $s_2$ are equal to $\bb$ then  $(s_1s_2\colon r_1r_2) \not \in \cC$, $(s_2 \ccolon r_2) \notin \cC$ and item 5c) of Proposition~\ref{bounds} is fulfilled, reaching a code with $k=\sigma$ and $r=\sigma+6$.
\end{proof}

For a generic Hadamard $\qq$-code, the range of rank values as well as the range of values given by the dimension of the kernel depends on the specific shape of the code. However, summing up the Proposition~\ref{bounds} and all results in this section about constructions, we can establish a tight upper and lower bound for the values of the rank and dimension of the kernel for Hadamard $\qq$-codes.  The next theorem gives these bounds, which improve the ones previously given in \cite{rio-2013}. Further, in this section we have given constructions of Hadamard $\qq$-codes covering all allowable values for the pair rank, dimension of the kernel.

\begin{theo}\label{fites}
Let $C$ a Hadamard $\qq$-code of length $2^m$ and $|T(\cC)|=2^\sigma$, where $T(\cC)$ is the subgroup of elements in $\cC$ of order two; $| \cC / \cA(\cC)|=2^\upsilon$; $|\cA(\cC) / T(\cC)|=2^\tau$; $|\cR(\cC) / T(\cC)|=2^\tauB$; $|\cC / T(\cC)|=2^{\tau+\upsilon}$ and $m+1=\sigma+\tau+\upsilon$. Then the rank $r$ and the dimension of the kernel $k$ of $C$ satisfy the following conditions.
\begin{enumerate}
\item The values of the dimension of the kernel are $1\not= m+1-k\in \{0, 4,\tau-1,\tau,\tau+1\}$. The specific case $m+1-k=0$ is obtained in codes where $\tauB\leq 1$ or in codes of shape 5. The specific case $m+1-k=4$ is obtained in codes of shape 5.
\item
\begin{enumerate}
\item If $m+1-k=0$ then we have $r-(m+1)=0$,
\item If $m+1-k=4$ and $\upsilon=2$ then we have $r-(m+1)=2$,
\item If $m+1-k=\tau-1\geq 2$ then we have $r-(m+1)=\binom{\tau-1}{2}$,
\item If $m+1-k=\tau \geq 2$ then we have $r-(m+1)=\binom{\tau}{2}$,
\item If $m+1-k=\tau+1$ and $\tauB <= 1$ then we have $r-(m+1)=\tau$.
\item If $m+1-k=\tau+1$ and $\tauB=\tau-1 \ge 2$ then we have $r-(m+1)\in \{\binom{\tau-1}{2},\ldots \binom{\tau}{2}+1\}$.
\item If $m+1-k=\tau+1$ and $\tauB=\tau \ge 2$ then we have $r-(m+1)\in \{\binom{\tau}{2}+1,\ldots \binom{\tau+1}{2}\}$.
\end{enumerate}
\end{enumerate}
\end{theo}

\begin{example}\label{ex:1}
{\rm
The following example shows constructions of codes of length $n=2^m=2^7=128$, with $\tau=3 \ge \tauB=2 \ge 2, \upsilon=1$ (item 4 of Proposition~\ref{bounds}) and  $\sigma=4$. The resulting codes are of shape 2 and, before the Gray map, subgroups of $\cQ_8^{32}$.  All possible pairs of rank and dimension of the kernel are presented:

Let $\overline{r_1}, \overline{r_2}, \overline{r_3}, \overline{x_1} \in Q_8^{16}$ be the vectors:
$$
\begin{array}{cccccccccccccccc}
\overline{r_1} = ( \ba& \ba& \ba& \ba& \ba& \ba& \ba& \ba& \ba& \ba& \ba& \ba& \ba& \ba& \ba& \ba) \\
\overline{r_2} = ( \ba& \ba& \ba^3& \ba^3& \ba& \ba& \ba^3& \ba^3& \one& \one& \ba^2& \ba^2& \one& \one& \ba^2& \ba^2 ) \\
\overline{r_3} = ( \ba& \ba^3& \ba& \ba^3& \one& \ba^2& \one& \ba^2& \ba& \ba^3& \ba& \ba^3& \one& \ba^2& \one& \ba^2 ) \\
\overline{x_1} = ( \one& \one& \one& \one& \one& \one& \one& \one& \one& \one& \one& \one& \one& \one& \one& \one)\\
\overline{x_2} = (\ba^2 &\ba^2 &\ba^2 &\ba^2 &\ba^2 &\ba^2 &\ba^2 &\ba^2 &\ba^2 &\ba^2 &\ba^2 &\ba^2 &\ba^2 &\ba^2 &\ba^2 &\ba^2 )
\end{array}
$$
and now, take the following vectors in $Q_8^{32}$:
$$
\begin{array}{cc}
r_1 = ( \overline{r_1}, \overline{r_1})  \\
r_2 = ( \overline{r_2},\overline{r_2} ) \\
r_3 = ( \overline{r_3} , \overline{r_3}  ) \\
x_1 = ( \overline{x_1}, \overline{x_2})
\end{array}
$$

Let $y_1, y_2, y_3, y_4, y_5, y_6\in Q_8^{16}$ be the vectors:
$$
\begin{array}{cccccccccccccccc}
y_1= ( \bb& \bb& \bb& \bb& \bb& \bb& \bb& \bb& \bb& \bb& \bb& \bb& \bb& \bb& \bb& \bb)\\
y_2= (\ba\bb& \ba\bb&\ba\bb&\ba\bb&\ba\bb&\ba\bb&\ba\bb&\ba\bb&\ba\bb&\ba\bb&\ba\bb&\ba\bb&\ba\bb&\ba\bb&\ba\bb&\ba\bb)\\
y_3= ( \bb& \bb& \bb& \ba\bb& \bb& \bb& \bb& \ba\bb& \bb& \bb& \bb& \ba\bb& \bb& \bb& \bb& \ba\bb)\\
y_4= ( \bb& \bb& \bb& \bb& \bb& \bb& \bb& \bb& \bb& \bb& \bb& \ba\bb& \bb& \bb& \bb& \ba\bb)\\
y_5= ( \bb& \bb& \bb& \bb& \bb& \bb& \bb& \bb& \bb& \bb& \bb& \bb& \bb& \bb& \bb& \ba\bb)\\
y_6= ( \bb& \bb& \bb& \bb& \bb& \bb& \bb& \bb& \bb& \bb& \bb& \bb& \ba\bb& \ba\bb& \ba\bb& \ba\bb)\\
\end{array}
$$
The codes with all possible pairs of values rank,dimension of the kernel are generated by $r_1, r_2, r_3$ and $s_1$ which is taken following Theorem~\ref{constructions}. We show the vector $s_1$ and the values of the pair rank, dimension of the kernel.

When $s_1=(y_1,y_1)$ the constructed code has $k=6,r=9$.

When $s_1=(y_2,y_1)$ the constructed code has $k=5,r=11$.

When $s_1=(y_3,y_3)$ the constructed code has $k=4,r=12$.

When $s_1=(y_4,y_4)$ the constructed code has $k=4,r=11$.

When $s_1=(y_5,y_5)$ the constructed code has $k=4,r=10$.

When $s_1=(y_6,y_6)$ the constructed code has $k=4,r=9$.
}
\end{example}

\begin{example}\label{ex:2}
{\rm
The following example shows constructions of codes of length 32, with $\tau=\tauB=2$, $\upsilon=1$
(item 4 of Proposition IV.3) and $\sigma = 3$. The resulting codes are of shape 3 and, before the Gray map, subgroups of
 $\Z_4^4\cQ_8^{6}$.  All possible pairs of rank and dimension of the kernel are presented.

 Take the following vectors in $\Z_4^4\cQ_8^{6}$:
$$
\begin{array}{cccccccccc}
r_1 = ( 0& 2& 0& 2& \one& \ba^2& \ba& \ba& \ba& \ba ) \\
r_2 = ( 0& 0& 2& 2& \ba& \ba& \one& \ba^2& \ba& \ba^3 )
\end{array}
$$
The codes with all possible pairs of values rank, dimension of the kernel are generated by $r_1, r_2$ and $s_1$ which is taken following Theorem~\ref{constructions}. We show the vector $s_1$ and the values of the pair rank, dimension of the kernel.

When $s_1= ( 1,1,1,1, \bb, \bb, \bb, \bb, \bb, \bb )$ the constructed code has $k=4,r=7$.

When $s_1= ( 1,1,1,1, \bb, \ba\bb, \bb, \ba\bb, \bb, \ba\bb )$ the constructed code has $k=3,r=9$.

When $s_1= ( 1,1,1,1, \bb, \ba\bb, \bb, \bb, \bb, \bb )$ the constructed code has $k=3,r=8$.
}
\end{example}

\end{document}